\mag=\magstep1 
\input epsf 
\documentclass[10pt,a4paper,dviout]{amsart} 
\usepackage{amsmath}	
\usepackage{amssymb,latexsym} 
\usepackage{graphicx} 
\usepackage{color}
\usepackage{datetime}
\usepackage{comment}

\def\Cal{\mathcal}

\def\Ch{\operatorname{Ch}} 
 
\def\<<{\langle } 
\def\>>{\rangle }

\numberwithin{equation}{section} 
\newtheorem{theorem}{Theorem}[section] 
\newtheorem{proposition}[theorem]{Proposition} 
 
\newtheorem{definition}[theorem]{Definition}

\newtheorem{lemma}[theorem]{Lemma}

\def\<{\langle} 
\def\>{\rangle} 
\excludecomment{donotforget}
\allowdisplaybreaks[4]

\begin{document}

\title[Fundamental group of non-singular locus of Lauricella's $F_C$]
{Fundamental group of non-singular locus of Lauricella's $F_C$}

\author{Tomohide Terasoma}
\address[Terasoma]{
Graduate school of Mathematical Sciences, 
The University of Tokyo, Tokyo 153-8914 Japan}
\email{terasoma@ms.u-tokyo.ac.jp}


\maketitle

\medskip

\begin{abstract}
In this paper, we give a set of generators and relations of the fundamental
group $\pi_1(\overline{Y_n})$ of the non-singular locus $\overline{Y_n}$
of Lauricella's hypergeometric function $F_C$.
\end{abstract}

\tableofcontents

\section{Introduction and motivation}
The Lauricella hypergeometric function $F_C^{(n)}$ of $n$ variable defined by
\begin{align*}
&F_C^{(n)}(a,b;c_1,\dots, c_n;z_1, \dots, z_n)
\\
=&
\sum_{m_1, \dots, m_n\in \bold Z_{\geq 0}}\dfrac{(a,m_1+\cdots +m_n)(b,m_1+\cdots +m_n)z_1^{m_1}\cdots z_1^{m_1}}
{(c_1,m_1)\cdots (c_n,m_n)m_1!\cdots m_n!},
\end{align*}
and has the following integral expression (\cite{G}):
$$
(\text{const.})\int \prod_{k=1}^n t_k^{-c_k}\cdot 
(1-\sum_{k=1}^n t_k)^{\sum c_k-a-n}\cdot (1-\sum_{k=1}^n\dfrac{z_k}{t_k})^{-b}
dt_1\cdots dt_n.
$$
Using Caylay technique \cite{GKZ}, the function $F_C^{(n)}$ is locally holomorphic on$(z_i)\in (\bold C^{\times})^{n}$
if the toric hypersurface
$$
\{((t_i)_i,\lambda)\in (\bold C^{\times})^{n+1}
\mid\lambda(1-\sum_{k=1}^n t_k)+(1-\sum_{k=1}^n \dfrac{z_i}{t_k})=0\}
$$
is non-degenerate for Newton polyhadra. For non-degneracy condition, see \cite{T}.

Since the non-degeneracy condition for a proper Newton polyhedra is equal to the smoothness of the
varieties
\begin{align*}
&\{\lambda(1+\sum_{i\in I}t_i)+\sum_{j\in J}\dfrac{a_j}{t_j}=0\},
\\
&\{\lambda(\sum_{i\in I}t_i)+1+\sum_{j\in J}\dfrac{a_j}{t_j}=0\}
\end{align*}
for $I, J\subset \{1, \dots, n\}$ and $I\cap J=\emptyset$.
Therefore the non-degeneracy condition is equivalent to the smoothness to
the open face.
Using Jacobian criterion, the singular locus is defined by
\begin{align*}
\begin{cases}
1-\sum_{k=1}^n t_k=0,
\\
\lambda-\dfrac{z_i}{t_i^2}=0,
\\
\lambda(1-\sum_{k=1}^n t_k)+(1-\sum_{k=1}^n \dfrac{z_i}{t_k})=0.
\end{cases}
\end{align*}
By setting $\mu^2=\lambda, x_i^2=z_i$,
and using the first and the second equations, $\mu$ is obtained by
$$t_i\mu=\epsilon_ix_i,\quad
\mu-\sum_{i=1}^n \epsilon_ix_i=
\mu(1-\sum_{i=1}^n t_i)=0.
$$
Here $\epsilon_i\in \{-1,1\}$. Again, using the first and the second equations, the third equation is equal to
\begin{align*}
0&=1-\sum_{i=1}^n \dfrac{x_i^2}{t_i}
=1-\lambda  \sum_{i=1}^n  t_i=1-\lambda=(1+\mu)(1-\mu)
\\
&
=(1+\sum_{i=1}^n \epsilon_ix_i)
(1-\sum_{i=1}^n \epsilon_ix_i).
\end{align*}
Therefore under the $\mu_2^n$-covering map, 
$$
\bold C^n=\{(x_1,\dots, x_n)\}\ni (x_i)_i \mapsto (x_i^2)_i=(z_i)_i\in
\bold C^n=\{(z_1,\dots, z_n)\}.
$$
the pull back $Y_n$ of $\overline{Y_n}$ is given by 
$$
Y_n=\{(x_i)_i\mid \prod_{k=1}^n x_k \prod_{\epsilon_i \in \{-1,1\}}(1-\sum_{i=1}^n \epsilon_ix_i)\neq 0\}.
$$
See also \cite{HT}.
Therefore $\overline{Y_n}\subset \{(z_1, \dots, z_n)\}$ is isomorphic to $Y_n/\mu_2^n$.

In the study of monodromy of hypergeometric function of type $F_C$,
it is a basic problem to give an expression of 
the fundamental group of $\overline{Y_n}$.
The generator and relations of the fundamental group for 
$n=2$ and $n=3$ is determined in \cite{GK}.
We prove the following presentation of the fundamental group which is conjectured in
\cite{GK}.
\begin{theorem}[Main Theorem, see Theorem \ref{main theorem} and Proposition 
\ref{restatement of commutativity rel}]
\label{main theorem introduction}
The fundamental group of $\overline{Y_n}$ is generated by elements
$\Gamma_0, \Gamma_1,\dots, \Gamma_n$ with the relations
\begin{align*}
[\Gamma_i,\Gamma_j]=1,  \quad (1\leq i,j\leq n),
\quad 
(\Gamma_0\Gamma_i)^2=(\Gamma_i\Gamma_0)^2, \quad (1\leq i\leq n),
\end{align*}
and
\begin{equation*}
[M(I)^{-1}\Gamma_0 M(I),M(J)^{-1}\Gamma_0 M(J)]=1
\end{equation*}
for all subsets $I$ and $J$ of $\{1, \dots, n\}$ 
satisfying $I\cap J=\emptyset, I\neq \emptyset, J\neq \emptyset$ and $\#I+\#J\leq n-1$.
Here we set $M(I)=\prod_{i\in I}\Gamma_i$.
\end{theorem}
For the proof of this theorem, we use a cell complex constructed by Salvetti \cite{S}, 
which is homotopic to the complement of a hyperplane arrangement
in $\bold C^N$ and stable under a group action.
The author is grateful for discussions with Y. Goto and K. Matsumoto 
in ``Workshop on Special Varieties in Tambara, 2017'', in Tambara International Seminar House.

\section{Recall of a result of Salvetti}
We recall a construction of $2$-skeleton of a cell complex which is homotopic
to the complement of real hyperplane arrangement.
A finite set $\Cal H=\{H_i\}_{i\in I}$ of complex hyperplanes in $\bold C^n$ is called a
hyperplane arrangement. In this paper, we are interested in the topological space
$$
Y=Y(\Cal H)=\bold C^n-\cup_{i\in I}H_i.
$$
A hyperplane arrangement is called a real hyperplane arrangement if the defining equations
of $H_i$ is defined over $\bold R$ for all $i\in I$.
For a 
real hyperplane arrangement $\Cal H$, 
we set $H_{i,\bold R}=H_i\cap \bold R^n$.
The set $\{H_{i,\bold R}\}_{i\in I}$
is denoted by $\Cal H_R$.
A subset of $\bold R^n$ which can be obtained by the intersection of finite number
of $H_{i,\bold R}$'s is simply called a linear subset of $\Cal H_{\bold R}$.
As a special case, the total space $\bold R^n$ is an $n$-dimensional linear subset.
Let $L$ be an $i$-dimensional linear subset of $\Cal H_\bold R$.
A connected component of the complement of the union of proper linear subsets of $L$ in
$L$ is called an $i$-chamber of $\Cal H_{\bold R}$ and the set of $i$-chamber is denoted by
$\Ch_i=\Ch_i(\Cal H_\bold R)$.
Each $i$-chamber is a convex set.

We define the dual cell complex of $\Cal H_{\bold R}$ as follows.
For each $i$-dimensional chamber $\sigma$, we choose a vertex $v_{\sigma}$ in
the interior of $\sigma$. 
The set of $0$-cell of the dual cell complex is given by $D_{\sigma}=v_{\sigma}$, where
$\sigma$ is an $n$-chamber.

Let $\tau$ be an $(n-1)$-chamber. Then there exist exactly two
$n$-chambers
$\tau_1$ and $\tau_2$
 such that $\overline{\tau_i}\supset \overline{\tau}$ for $i=1,2$.
Here $\overline{\tau}$ is the closure of $\tau$ in $\bold R^n$.
We consider 1-cell $D_{\tau}$ by considering the union of segments $\Delta(v_{\tau_1},v_{\tau})$ 
and $\Delta(v_{\tau_2},v_{\tau})$.
We continue this procedure to define $2$-cell $D_{\sigma}$ attached to $(n-2)$-dimensional
chamber as follows. If $\sigma_1, \sigma_2$ and $\sigma$ are $n, (n-1)$ and $(n-2)$-chambers, 
such that 
$$
\overline{\sigma_1}\supset \overline{\sigma_2}\supset \overline{\sigma}.
$$
A sequence $F=F(\sigma_1,\sigma_2,\sigma)$ as above is called a (descending) flag of length $3$.
The triangle $\Delta(v_{\sigma_1},v_{\sigma_2},v_{\sigma})$ is called the dual flag $F^*$ of $F$.
The union of dual flags containing $v_{\sigma}$ is called the $2$-dimensional dual cell $D_{\sigma}$ of 
$\sigma$.

\begin{figure}[htbp]
\includegraphics[width=4.5cm]{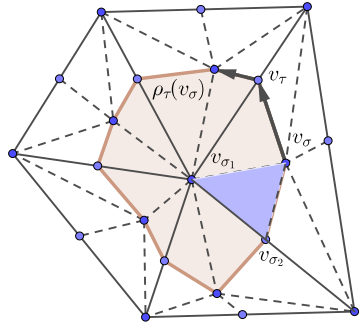}
\caption{Dual cell}
\label{dual cell}
\end{figure}

We recall the construction of the $2$-skeleton $X_2$ of the cell complex $X$ after Salvetti \cite{S}, 
which is homotopy equivalent
to the space $Y=Y(\Cal H)$.
The set $C_0(X)$ of $0$-cell in $X$ is the set $\{\widetilde{D_\sigma}\}_{\sigma\in \Ch_n}$
of the copy $\widetilde{D_{\sigma}}$ of $D_{\sigma}$.

The set $C_1(X)$ of $1$-cell consists of $\widetilde D_{\sigma,\tau}$
for $\sigma \in \Ch_n, \tau \in \Ch_{n-1}$ such that $\overline{\sigma}\supset \tau$.
The $n$-chamber lying on the opposite side of $\sigma$ with respect to the $(n-1)$-chamber $\tau$ 
is denoted by $\rho_{\tau}(\sigma)$.
The attaching map $\partial \widetilde D_{\sigma,\tau}\to X_0$ is given by
connecting two points $\sigma$ and $\rho_{\tau}(\sigma)$. The $1$-cell $\widetilde D_{\sigma,\tau}$
is called an arrow from $\sigma$ to $\rho_{\tau}(\sigma)$. The composite of several arrows 
compatible with the directions is called an oriented path.

The set $C_2(X)$ of $2$-cell consists of 
$\widetilde D_{\sigma,\tau}$ for $\sigma \in \Ch_n, \tau \in \Ch_{n-2}$ such that 
$\overline{\sigma}\supset \overline{\tau}$.
The $n$-chamber lying on the opposite side of $\sigma$ with respect to the $(n-2)$-chamber $\tau$ 
is denoted by $\rho_{\tau}(\sigma)$ and the vertex in $\rho_{\tau}(\sigma)$ is denoted by
$\rho_{\tau}(v_\sigma)$ (see Figure \ref{bounding disc}).
\begin{figure}[htbp]
\includegraphics[width=4cm]{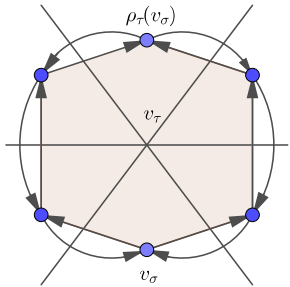}
\caption{Relations}
\label{bounding disc}
\end{figure}
Then there exist exactly two shortest paths from $v_\sigma$ to $\rho_{\tau}(v_\sigma)$.
The attaching map $\partial \widetilde D_{\sigma,\tau}\to X_1$ is given by
bounding the two shortest paths from $v_\sigma$ to $\rho_{\tau}(v_\sigma)$ 
(see Figure \ref{bounding disc}).

\begin{proposition}[Salvetti]
The natural inclusion $X_2 \to Y$ induces an isomorphism of fundamental groups
$$
\pi_1(X_2) \to \pi_1(Y).
$$
As a consequence, the fundamental groupoid is generated by $\widetilde D_{\tau_1,\tau_2}$
for 
$\tau_1\in \Ch_n, \tau_2\in \Ch_{n-1}, \overline{\tau_1}\supset\overline{\tau_2}$,
and the relation is given by $\widetilde D_{\sigma_1,\sigma_2}$
for $\sigma_1\in \Ch_n, \sigma_2\in \Ch_{n-2}, \overline{\sigma_1}\supset\overline{\sigma_2}$.
\end{proposition}

\section{$F_C$-hyperplane arrangement}
\subsection{The arrangement $\Cal H_n$}
For an element $\epsilon=(\epsilon_1,\dots, \epsilon_n)\in \{-1,1\}^n$,
we define a hyperplane $H_{\epsilon}$ by
$$
H_{\epsilon}:\epsilon_1x_1+\cdots +\epsilon_nx_n=1.
$$
We define $n$-dimensional $F_C$-arrangement $\Cal H_n$ by the union of the set of hyperplanes $\{H_\epsilon\}$
$(\epsilon\in \{-1,1\}^n)$ and that of coordinate hyperplanes 
$$
L_i:x_i=0, \quad (i=1, \dots, n).
$$
The following proposition is used to classify $(n-2)$-chambers in $\Cal H_n$.
\begin{proposition}
\label{codimension 2 linear space}
\begin{enumerate}
\item
Let $\epsilon,\epsilon'$ be elements in $\{-1,1\}^n$ 
such that $\#\{i\mid \epsilon_i\neq \epsilon'_i\}\geq 2$ and set 
$$
H_{\epsilon,\epsilon'}=H_{\epsilon}\cap H_{\epsilon'}.
$$
A hyperplane in $\Cal H$ containing $H_{\epsilon,\epsilon'}$ is equal to $H_{\epsilon}$ or 
$H_{\epsilon'}$.
\item
For an element $\epsilon$ in $\{-1,1\}^n$ and an integer $i$ with $1\leq i\leq n$,
we set
$$
H_{\epsilon,i}=H_{\epsilon}\cap L_i.
$$
A hyperplane in $\Cal H$ containing $H_{\epsilon,i}$ is equal to $L_i$, $H_{\epsilon}$ or 
$H_{g^{(i)}(\epsilon)}$. Here 
\begin{equation}
\label{reflection gi}
g^{(i)}(\epsilon_1, \dots, \epsilon_n)=(\epsilon_1,\dots,\overset{i}{-\epsilon_i},\dots, \epsilon_n).
\end{equation}
\item
Let $i,j$ be distinct integers such that $1\leq i,j \leq n$ and set 
$$
H_{i,j}=L_i\cap L_j.
$$
A hyperplane in $\Cal H$ containing $H_{i,j}$ is equal to $L_i$ or $L_j$. 
\end{enumerate}
\end{proposition}

\subsection{Group action}
On the space $Y$, the group $\mu_2^n=\{1,-1\}^n$ acts by
$$
g:\bold C^n\to \bold C^n:(x_1, \dots, x_n)\mapsto
(g_1x_1,\dots, g_nx_n)
$$
for $g=(g_1, \dots, g_n)\in \mu_2^n$. The group $\mu_2^n$ acts on the sets $\Ch_i$.
We can choose the set of vertex $\{v_{\sigma}\}_{\sigma\in \Ch_i}$ so that they are stable under the action of
$\mu_2^n$. 
\begin{lemma}
On the topological space $X_2$, the action of the group $\mu_2^n$ on $X_2$
is cell-wise and fixed point free. 
\end{lemma}
\begin{proof}
The group acts on $\Ch_n$ freely. Therefore it acts freely on the set of $0, 1$ and $2$-cells.
\end{proof}

\subsection{Cell complex for the quotient space}
We consider topological space $\overline{X_2}=X_2/\mu_2^n$. Then
$\overline{X_2}$ is a cell complex. 
We have the following proposition.
\begin{proposition}
The natural map 
$\pi_1(\overline{X_2}) \to \pi_1(Y/\mu_2^n)$
is an isomorphism.
\end{proposition}
We describe the cell complex $\overline{X_2}$ in this subsection.
We set 
$$
\bold R_{>0}=\{x\in \bold R\mid x> 0\},\quad
\bold R_{\geq 0}=\{x\in \bold R\mid x\geq 0\}.
$$
The subset of $i$-chambers in $\Ch_i$ contained in $\bold R_{\geq 0}$ is 
denoted by $\overline{\Ch}_{i}$.

The set $C_0(\overline{X})$ of $0$-cells in $\overline{X}$ is identified with 
$\{\widetilde D_\sigma\mid \sigma\in \overline{\Ch}_n\}$.
There are the following two kinds of $1$-cells in $\overline{X}$:
The image of $\widetilde D_{\sigma,\tau}$,  
$(\sigma \in \overline{\Ch}_{n}, \tau \in \overline{\Ch}_{n-1})$ such that
\begin{enumerate}
\item(type 1, non-closed one cell)
$\tau\subset H_{\epsilon}$,
$(\epsilon\in \{-1,1\}^n)$.
\item(type 2, closed one cell)
$\tau\subset L_i$,  $(1\leq i \leq n)$.
\end{enumerate}
There are three kinds of $2$-cells in $\overline{X_2}$:
The image of $\widetilde D_{\sigma,\tau}$,  
$(\sigma \in \overline{\Ch}_{n}, \tau \in \overline{\Ch}_{n-2})$ such that
\begin{enumerate}
\item(type 1, interior disc)
$\tau\subset H_{\epsilon}\cap H_{\epsilon'}$,
\item(type 2, boundary disc)
$\tau\subset H_{\epsilon}\cap L_i$,
\item(type 3, coordinate disc)
$\tau\subset L_i\cap L_j$.
\end{enumerate}

\begin{definition}
Let $\sigma$ be an element in $\overline{\Ch}_n$. 
we define height $h(v_{\sigma})$ of $v_{\sigma}=\widetilde D_{\sigma}$
by the number of hyperplanes of the form $H_{\epsilon}$ $(\epsilon \in \{-1,1\}^n)$
separating $\bold 0$ and $v_{\sigma}$. The number $h(v_{\sigma})$ is also denoted by $h(\sigma)$.
\end{definition}

\begin{proposition}
\begin{enumerate}
\item
A interior disc (type 1) is attached to four $1$-cells and contains four $0$-cells.
The shape of height is as follows.
\item
A boundary disc (type 2) is attached to six $1$-cells and contains three $0$-cells.
\item
A coordinate disc (type 3) is attached to two $1$-cells and contains one $0$-cells.
\end{enumerate}
\end{proposition}

We define spanning complex which is a slight generalization of spanning tree.
A $1$-cell $\widetilde D_{\sigma,\tau}$ is called a spanning $1$-cell if it is type 1
and $h(\sigma)+1=h(\rho_{\tau}(\sigma))$, i.e. $\rho_{\tau}(\sigma)$
is farer from the origin than $\sigma$.
A $2$-cell $\widetilde D_{\sigma,\tau}$ is called a spanning $2$-cell if it is type 1
and $h(\sigma)$ is the smallest among vertices contained in $D_{\tau}$.
The union of spanning $1$ and $2$-cells forms a sub cell complex $\Cal S$ of $\overline{X_2}$.
The complex $\Cal S$ is called the spanning complex of $\overline{X_2}$.
\begin{lemma}
The spanning complex $\Cal S$ is simply connected. 
\end{lemma}
\begin{proof}
It is identified with a $2$-skeleton of the dual cell complex of $\bold R^n_{>0}$
which is simply connected.
\end{proof}

We define $\overline{X_2}^{(s)}$ by obtaining contracting a subset $\Cal S\subset \overline{X_2}$ to a point $s$.
By the above proposition, we have
\begin{proposition}
The natural map 
$$
\pi_1(\overline{X_2}) \to \pi_1(\overline{X_2}^{(s)},s)
$$
is an isomorphism.
\end{proposition}
\begin{definition}
A $1$-cell $\widetilde D_{\sigma,\tau}$ in $\overline{X_2}$ is called a generator if it is
\begin{enumerate}
\item
type 1 and not spanning, or
\item
type 2.
\end{enumerate}
\end{definition}
A generator defines a closed path in $\overline{X_2}^{(s)}$.
Then the set of generator generates the group $\pi_1(\overline{X_2}^{(s)})$.

\subsection{Relations for type 1 and type 2}
\subsubsection{Type 1 relation}
First, we consider a type 1 $2$-cells 
$\widetilde D_{\sigma,\tau}$ in $\overline{X}$ with 
$\tau\subset H_{\epsilon}\cap H_{\epsilon'}$.
\begin{figure}[htb]
\includegraphics[width=5cm]{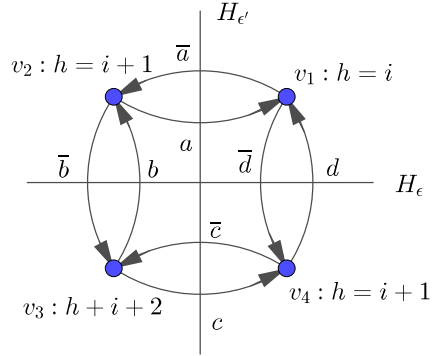}
\caption{Type 1 relation}
\label{type 1 rel}
\end{figure}
The arrows $\overline{a}, \overline{b}, \overline{c}$ and $\overline{d}$
are spanning $1$-cells and $a,b,c$ and $d$ define elements
in $\pi_1(\overline{X_2}^{(s)},s)$.
Their relations are given as
$$
a=c,\quad b=d,\quad ab=ba.
$$

\subsubsection{Type 2 relation}
Next we consider a type 2 $2$-cells $\widetilde D_{\sigma,\tau}$ in $\overline{X}$
as in Figure \ref{type 2 rel}.
\begin{figure}[htb]
\includegraphics[width=6cm]{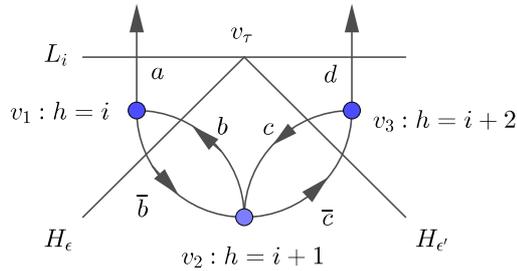}
\caption{Type 2 relation}
\label{type 2 rel}
\end{figure}
The arrows $\overline{b}$ and $\overline{c}$ are spanning $1$-cells and reduced to one point in $\overline{X_2}^{(s)}$.
We consider a $(n-2)$-chamber $\tau$ contained in $L_i$.
The relations beginning from $v_1=v_{\sigma_1}, v_2=v_{\sigma_2}$ and 
$v_3=v_{\sigma_3}$ are the following:
\begin{align*}
&\partial\widetilde D_{\sigma_1,\tau}:a=d, \\
&\partial\widetilde D_{\sigma_2,\tau}:ba=dc, \\
&\partial\widetilde D_{\sigma_3,\tau}:cba=dcb. 
\end{align*}
We can easily check the following proposition.
\begin{proposition}
The above relations are equivalent to  
\begin{equation}
\label{type 2 braid relation}
d=a, \quad c=a^{-1}ba, \quad (ab)^2=(ba)^2.
\end{equation}
\end{proposition}

We consider the above situation and set 
$\epsilon=(\epsilon_1, \dots, \epsilon_n)$ and
$\epsilon'=(\epsilon'_1, \dots, \epsilon'_n)$. 
Here $\epsilon'=g^{(i)}(\epsilon)$ where $g^{(i)}$ is defined as (\ref{reflection gi}).
By the definition of height, $v_1$ is the closest vertex from the origin.
Therefore we have $\epsilon_j=\epsilon'_j$ if $j\neq i$
and $\epsilon_i=1$ and $\epsilon'_i=-1$.

\subsection{Definition of $\gamma_i$ and their relations}
In this subsection, we define $\gamma_i$ and study their relations.
\begin{definition}
We define 
$\gamma_i=\widetilde D_{\sigma_0,\tau_i}$
for $i=0,1,\dots,n$ where
\begin{align*} 
& \sigma_0=\{(x_i)\in \bold R^n\mid
x_i>0 \quad (0\leq i\leq n), \quad  \sum x_i<1\},
\\
& \tau_0=\{(x_i)\in \bold R^n\mid
x_i>0 \quad (0\leq i\leq n), \quad  \sum x_i=1\},
\\
& \tau_i=\{(x_i)\in \bold R^n\mid
x_i>0 \quad (0\leq j\leq n, j\neq i), \quad  \sum x_i<1,  \quad x_i=0\}.
\end{align*}
Actually $\sigma_0$ is a chamber since if $x=(x_1, \dots, x_n)\in \sigma_0$
and $\epsilon\neq (1, \dots, 1)$ then
$$
\sum_i \epsilon x_i <\sum_i x_i <1,
$$
and the point $x$ is not contained in $H_{\epsilon,\bold R}$.
\end{definition}
By previous subsection, we have
\begin{align}
\label{fundamental 1}
&[\gamma_i,\gamma_j]=1,  \quad (1\leq i,j\leq n),
\\
\nonumber
&(\gamma_0\gamma_i)^2=(\gamma_i\gamma_0)^2, \quad (1\leq i\leq n).
\end{align}
\begin{proposition}
\label{equalities in pi 1}
\begin{enumerate}
\item
Under the notation of figure 1, we have 
$a=\gamma_i$ in $\pi_1(\overline{X_2}^{(s)})$.
\item
Let $\tau_1$ and $\tau_2$ be two elements in $\overline{\Ch}_{n-1}$ 
contained a common hyperplane $H_{\epsilon,\bold R}$.
Suppose that $1$-cells
$\widetilde D_{\sigma_1,\tau_1}$ and
$\widetilde D_{\sigma_2,\tau_2}$ are generators.
Then the paths obtained by them are
homotopic to each other.
These paths defines a common element in $\pi_1(\overline{X_2}^{(s)},s)$ which is denoted by
$\gamma_{\epsilon}$.
\item
For $\epsilon \in \{-1,1\}^n, \epsilon\neq (-1, \dots, -1)$, we set
$$
S(\epsilon)=\{i \mid 1\leq i\leq n, \epsilon_i=-1 \},
\quad m_{\epsilon}=\prod_{i\in S(\epsilon)}\gamma_i.
$$
Then we have 
\begin{equation}
\label{expression of inner path} 
\gamma_{\epsilon}={m_\epsilon}^{-1}\gamma_0 m_\epsilon.
\end{equation}

\end{enumerate}
\end{proposition}
\begin{proof}
(1) We use the first relation (\ref{type 2 braid relation})
iteratively and have the statement.
(2) This follows from the relations obtained by a type 1 $2$-cells.
(3) Let $\epsilon$ be an element in $\{-1,1\}^n$ and $\epsilon\neq (-1, \dots, -1)$. 
We set $S(\epsilon)=\{i_1,\dots, i_k\}$.
We consider a chain $\epsilon^{(0)},\dots, \epsilon^{(k)} \in \{-1,1\}^n$ defined by
$$
\epsilon^{(0)}=(1,\dots, 1), \epsilon^{(1)}=g^{(i_1)}(\epsilon^{(0)}), 
\epsilon^{(2)}=g^{(i_2)}(\epsilon^{(1)}),
\dots, \epsilon^{(k)}=g^{(i_k)}(\epsilon^{(k-1)}).
$$
Then $\epsilon^{(k)}=\epsilon$.
\begin{lemma}
\label{desending ind lemme}
$H_{\epsilon^{(j)}}\cap \bold R^n_{>0}\neq \emptyset$, $(j=0, \dots, k)$
and 
$$
H_{\epsilon^{(j)}}\cap H_{\epsilon^{(j-1)}}\cap \{(x_l)\in \bold R^n\mid
x_l>0 \text{ for }l\neq i_j \}\neq \emptyset \quad (i=1, \dots, k).
$$
\end{lemma}
\begin{proof}[Proof of Lemma \ref{desending ind lemme}]
By descending induction, it is enough to prove the lemma for $j=k$.
We set
$\epsilon^{(k)}=(\epsilon_1, \dots, \epsilon_n)$.
First statement holds since $\epsilon_j=1$ for some $j$.
We prove the second statement. 
Since $\epsilon_{i_k}=-1$,
there exsists $j\neq i_k$ such that $\epsilon_j=1$.
Therefore the equation
$$
x_{i_k}=0,
\epsilon_1 x_1+\cdots +\widehat{\epsilon_{i_k}x_{i_k}}+\cdots +\epsilon_nx_n=1.
$$
has a solution satisfying $x_l>0$ for $l\neq i_j$.
\end{proof}
By applying the second relation in (\ref{type 2 braid relation}) iteratively, we have statement (3).
\end{proof}
Using Proposition \ref{equalities in pi 1} (3) and the relation of type 1, we have the following theorem.
\begin{theorem} 
\label{satisfy type 2 relation in pi1}
We have
\begin{equation}
\label{fundamental 2}
[{m_\epsilon}^{-1}\gamma_{0}{m_\epsilon},
{m_{\epsilon'}}^{-1}\gamma_{0}{m_{\epsilon'}}]=1.
\end{equation}
for 
$\epsilon, \epsilon' \in \{-1,1\}^n$ and
$
H_{\epsilon,\bold R}\cap H_{\epsilon',\bold R}\cap R_{>0}^n\neq \emptyset.
$
\end{theorem}
\section{Fundamental relation}
\subsection{Main theorem}
In this section, we prove the following theorem.
\begin{theorem}
\label{main theorem}
The relations 
(\ref{fundamental 1}) and
(\ref{fundamental 2}) are fundamental relations 
for $\pi_1(\overline{X_2}^{(s)},s)$
with generators $\gamma_0$ and $\gamma_i$ $(1\leq i\leq n)$.
\end{theorem}
We define $G$ as a group generated by $\Gamma_0$ and $\Gamma_i$
$(1\leq i \leq n)$ with the relations 
\begin{align}
\label{model fundamental 1}
&[\Gamma_i,\Gamma_j]=1,  \quad (1\leq i,j\leq n),
\\
\nonumber
&(\Gamma_0\Gamma_i)^2=(\Gamma_i\Gamma_0)^2 \quad (1\leq i\leq n),
\end{align}
and
\begin{equation}
\label{model fundamental 2}
[{M_\epsilon}^{-1}\Gamma_{0}{M_\epsilon},
{M_{\epsilon'}}^{-1}\Gamma_{0}{M_{\epsilon'}}]=1.
\end{equation}
for $H_{\epsilon,\bold R}\cap H_{\epsilon',\bold R}\cap R_{>0}^n\neq \emptyset.$
Here we set
$M_{\epsilon}=\prod_{i\in S(\epsilon)}\Gamma_i$.
We define group homomorphisms 
$$
\varphi:G\to \pi_1(\overline{X_2}^{(s)},s)
\text{ and }
\psi:\pi_1(\overline{X_2}^{(s)},s)\to G,
$$
which are inverse to each other.

\subsubsection{The definition of $\varphi$}
We define $\varphi$ by $\varphi(\Gamma_i)=\gamma_i$
for $i=0,1,\dots, n$. We check that fundamental relations of
$G$ are satisfied in $\pi_1(\overline{X_2}^{(s)},s)$.
The relation (\ref{model fundamental 1})
is satisfied by the definition of $\varphi$.
By the definition of $\varphi$,
we have
$$
\varphi({M_\epsilon}^{-1}\Gamma_{0}{M_\epsilon})
={m_\epsilon}^{-1}\gamma_{0}{m_\epsilon}.
$$
Thus the relation (\ref{model fundamental 2}) is satisfied by
Theorem \ref{satisfy type 2 relation in pi1}.
\subsubsection{The definition of $\psi$}
The group $\pi_1(\overline{X_2}^{(s)},s)$
is generated by type 1 non-spanning arrow $\gamma_{\epsilon,\tau}$
and type 2 generators $\gamma_{i,\tau}$
with the relation of type 1, type 2 and type 3 relations.
We set
$$
\psi(\gamma_{\epsilon,\tau})=M_{\epsilon}^{-1}\Gamma_0M_{\epsilon},\quad
\psi(\gamma_{i,\tau})=\Gamma_i.
$$
Type 1 and type 3 relation are satisfied by the fundamental relations of $G$.
The first relations of (\ref{type 2 braid relation})
is easy to check.
The second relation is obtained by the relation between $\epsilon$
and $g^{(i)}(\epsilon)$.
We check the third relation of (\ref{type 2 braid relation}) by
using
$\psi(a)=\Gamma_i, \psi(b)=M_{\epsilon}^{-1}\Gamma_0M_{\epsilon}$.
Since $\Gamma_i$ and $M_{\epsilon}$ are commutative in $G$, we have
\begin{align*}
\psi(abab)=
\Gamma_i\cdot M_{\epsilon}^{-1}\Gamma_0M_{\epsilon}\cdot 
\Gamma_i\cdot M_{\epsilon}^{-1}\Gamma_0M_{\epsilon} =
M_{\epsilon}^{-1}\Gamma_i\Gamma_0
\Gamma_i\Gamma_0M_{\epsilon}
\end{align*}
and 
\begin{align*}
\psi(baba)=
M_{\epsilon}^{-1}\Gamma_0M_{\epsilon}\cdot 
\Gamma_i\cdot M_{\epsilon}^{-1}\Gamma_0M_{\epsilon}\cdot \Gamma_i
=M_{\epsilon}^{-1}\Gamma_0 
\Gamma_i\Gamma_0\Gamma_iM_{\epsilon}.
\end{align*}
Thus we have the equality $\psi(abab)=\psi(baba)$.
Thus the homomorphism $\psi$ is well defined.
\begin{proof}[Proof of Theorem \ref{main theorem}]
By the definition of $\varphi$ and $\psi$, we see that the homomorphisms 
$\psi$ and $\varphi$ are inverse to each other.
\end{proof}
\subsection{Simplification}
We modify the relation of (\ref{model fundamental 2}) 
and get the simpler form cited in \cite{GK}.
By Theorem \ref{main theorem} and the following proposition, we get
Main Theorem \ref{main theorem introduction}.
\begin{proposition}
\label{restatement of commutativity rel}
For a subset $I$ of $\{1, \dots, n\}$ we set $M(I)=\prod_{i\in I}\Gamma_i$.
Under the relation (\ref{model fundamental 1}),
the relation (\ref{model fundamental 2}) for 
$H_{\epsilon,\bold R}\cap H_{\epsilon',\bold R}\cap R_{>0}^n\neq \emptyset$
is equivalent to the following set of relations.
\begin{quote}
\begin{equation}
\label{simplification}
[M(I)^{-1}\Gamma_0 M(I),M(J)^{-1}\Gamma_0 M(J)]=1
\end{equation}
for all $I,J$ satisfying $I\cap J=\emptyset, I\neq \emptyset, J\neq \emptyset$ and $\#I+\#J\leq n-1$.
\end{quote}
\end{proposition}
\begin{proof}
Throughout this proof we assume the commutativity of $\Gamma_1, \dots, \Gamma_n$.
First we assume the condition (\ref{simplification})
and prove the relation (\ref{model fundamental 2}).
We set $K=S(\epsilon)\cap S(\epsilon')$.
By the definition of $M_{\epsilon}$ and the commutativity of $\Gamma_i$, 
the condition (\ref{model fundamental 2}) can be rewrite as
\begin{equation}
\label{reduction from cell complex}
[{M^*_\epsilon}^{-1}\Gamma_{0}{M^*_\epsilon},
{M^*_{\epsilon'}}^{-1}\Gamma_{0}{M^*_{\epsilon'}}]=1.
\end{equation}
where 
$M_{\epsilon}^*=\prod_{i\in S(\epsilon)-K}\Gamma_i$
and $M_{\epsilon'}^*=\prod_{i\in S(\epsilon')-K}\Gamma_i$.
This is one of the conditions in 
(\ref{simplification})
by setting $I=S(\epsilon)-K$ and $J=S(\epsilon')-K$.
We check that $I$ and $J$ satisfies the required conditions.
The condition $I\cap J=\emptyset$ is clear.
If $M_{\epsilon}^*=\emptyset$, then $S(\epsilon)\subset S(\epsilon')$
and this contradicts to the condition 
$H_{\epsilon,\bold R}\cap H_{\epsilon',\bold R}\cap \bold R_{>0}^n\neq \emptyset$.
If $\#I+\#J=n$, then $\epsilon'=-\epsilon$. This also contradicts to the condition for $\epsilon$ and $\epsilon'$
since $H_{\epsilon,\bold R}\cap H_{-\epsilon,\bold R}=\emptyset$.

Next we assume the condition (\ref{model fundamental 2})
and prove the relation (\ref{simplification}).
Let $I$ and $J$ be subsets in $\{1, \dots, n\}$ satisfying the condition of
(\ref{simplification}). We define
$\epsilon=(\epsilon_1,\cdots, \epsilon_n)$
and
$\epsilon'=(\epsilon'_1,\cdots, \epsilon'_n)$
by
$$
\epsilon_i=\begin{cases}
1\quad (i\notin I),
\\
-1\quad (i\in I),
\end{cases}\quad
\epsilon'_i=\begin{cases}
1\quad (i\notin J),
\\
-1\quad (i\in J).
\end{cases}
$$
Then the relation (\ref{reduction from cell complex})
 becomes the relation (\ref{simplification}).
We check the condition $H_{\epsilon,\bold R}\cap H_{\epsilon',\bold R}\cap \bold R^n_{>0}\neq \emptyset$.
We set $K=\{1, \dots, n\}-(I\cup J)$.
Then we have $K\neq \emptyset$. The system of equations
$$
\begin{cases}
\sum_{i\notin I}x_i-\sum_{i\in I}x_i=1,
\\
\sum_{j\notin J}x_j-\sum_{j\in J}x_j=1
\end{cases}
$$
is equivalent to
$$
\begin{cases}
\sum_{i\in K}x_i=1,
\\
\sum_{i\in I}x_i=\sum_{j\in J}x_j.
\end{cases}
$$
Thus it has a solution $x=(x_i)\in \bold R_{>0}^n$.
\end{proof}


\begin{thebibliography}{MSTYY}
\bibitem[OT]{Orlik-Terao} P. Orlik, H.Terao 
\textsl{Arrangement of hyperplanes}, 
Grundlehren der Mathematischen Wissenschaften, Springer, 1992. 
%
\bibitem[GK]{GK} Y. Goto, J. Kaneko, 
The fundamental group of the complement of the singular locus of Lauricella's $F_C$,   
arXiv:1710.09594.
%
\bibitem[G]{G} Y. Goto, 
Twisted cycles and twisted period relations for Lauricella's hypergeometric function $F_C$. 
Internat. J. Math. {\bf 24}  (2013),  no. 12, 1350094, 19 pp
%
\bibitem[GKZ]{GKZ} I.M. Gel'fand, A. V. Zelevinski, M. M. Kapranov, Equations of hypergeometric type 
and Newton polyhedra. (Russian) Dokl. Akad. Nauk SSSR {\bf 300}  (1988),  
no. 3, 529--534;  translation in  Soviet Math. Dokl.  {\bf 37}  (1988),  no. 3, 678--682
%
\bibitem[HT]{HT} 
R. Hattori, N. Takayama, The singular locus of Lauricella's $F_C$, 
J. Math. Soc. Japan,
{\bf 66}(3) (2014), 981-995.
%
\bibitem[S]{S} 
M. Salvetti. Topology of the complement of real hyperplanes in $\bold C^N$, Invent. Math.,
{\bf 88}(3), 1987), 603--618. 
%
\bibitem[T]{T}
T. Terasoma, Boyarsky principle for $D$-modules and Loeser's conjecture.  
\textsl{Geometric aspects of Dwork theory}. Vol. I, II,  909--930, Walter de Gruyter, Berlin, 2004. 

\end{thebibliography}
\end{document}